\documentclass{amsart}

\usepackage[margin=3cm]{geometry}

\usepackage{authblk}
\usepackage{amsaddr}
\usepackage{amsmath}
\usepackage{amsthm}
\usepackage{amsfonts} 
\usepackage{hyperref}
\usepackage[utf8]{inputenc}
\usepackage[T1]{fontenc}

\usepackage{cleveref}
\usepackage{tikz}
\usepackage{algorithm}
\usepackage[noend]{algpseudocode}

\newcommand{\AT}{AT}
\newcommand{\ch}{ch}
\newcommand{\col}{col}

\theoremstyle{plain}
\newtheorem{thm}{Theorem}[section]
\newtheorem{lem}[thm]{Lemma}
\newtheorem{prop}[thm]{Proposition}

\theoremstyle{definition}
\newtheorem{defn}[thm]{Definition}

\newtheorem{fact}[thm]{Fact}

\theoremstyle{remark}
\newtheorem*{rem}{Remark}

\begin{document}

\title{Schnyder woods and Alon-Tarsi number of planar graphs}
\author{Jakub Kozik and Bartosz Podkanowicz}
\thanks{The author / co-author of the publication received an incentive scholarship
from the funds of the program Excellence Initiative - Research University
at the Jagiellonian University in Kraków.}
\address{Theoretical Computer Science Department, Faculty of Mathematics and Computer Science, Jagiellonian University, Krak\'{o}w, Poland}
\email{jakub.kozik@uj.edu.pl, bartosz.podkanowicz@doctoral.uj.edu.pl}
\maketitle

\begin{abstract}
Thomassen in 1994 published a famous proof of the fact that the choosability of a planar graph is at most 5. 
Zhu in 2019 generalized this result by showing that the same bound holds for Alon-Tarsi numbers of planar graphs.
We present an alternative proof of that fact, derived from the results on decompositions of planar graphs into trees known as Schnyder woods.
It turns out that Thomassen's technique and our proof based on Schnyder woods have a lot in common.
We discuss and explain the prominent role that counterclockwise 3-orientations play in proofs based on both these approaches.
\end{abstract}

\section{Introduction}
Thomassen proved in \cite{thom5c} that the choice number of a planar graph is at most 5. 
This result is best possible as there exist planar graphs with choice number 5.
The first such examples have been constructed by Voigt \cite{Voigt93}.
The proof technique introduced by Thomassen in \cite{thom5c} 
    has been used in a number of follow-up papers.
In particular, it has been generalised 
    to derive analogous results for more restrictive variants of graph colorings.
E.g. Schauz in \cite{schauz2010paintability} proved that online choice number (or paintability) of a planar graph is at most 5 as well.
We are going to discuss yet another generalization by Zhu \cite{zhuat}.
\begin{thm}[\cite{zhuat}]
\label{thm:planarAT1}
    The Alon-Tarsi number of a planar graph is at most 5.
\end{thm}

The main contribution of the current paper is a derivation of the above theorem from the results on the decompositions of planar graphs into trees.
Such decompositions, called realizers, have been designed by Schnyder in \cite{schem} for the purpose of constructing succinct straight-line drawings of planar graphs.
Our proof of Theorem \ref{thm:planarAT1} is conceptually (and technically) simpler than the one by Zhu.
However, it turns out that both arguments are related.
We show that in some sense the structure that is central in our derivation is also implicitly used in the works of Thomassen and Zhu.

In Section \ref{sec:prelim} we prepare the tools to be used in the proof of Theorem \ref{thm:planarAT1}. 
We recall (and extend) basic results on the Alon-Tarsi polynomial method and on Schnyder decompositions. 
Then, in Section \ref{sec:proof} we present our proof. 
Finally, in Section \ref{sec:Thom} we discuss its relation to the original proof of Thomassen for the choice number.

\section{Preliminaries}
\label{sec:prelim}
\subsection{Alon-Tarsi method}
For a polynomial $Q$, let $\alpha(Q)$ be the minimum $k$ such that there is a monomial $m$ that occurs in $Q$ with a nonzero coefficient, for which $\deg(Q)=\deg(m)$ and the maximum degree of any single variable in $m$ is at most $k$.
Note that, for a nonzero polynomials $P$ and $R$, we have $\alpha(P) \leq \alpha(P\cdot R)$.
For graph $G= (V,E)$ with $V = \{v_1, v_2, \dots, v_n\}$, \emph{graph polynomial} $f_G$ is defined as
\[
    f_G(x_1,x_2,\dots, x_n) = \prod_{i<j \wedge \{v_i, v_j\} \in E}(x_{i} - x_{j}).
\]
This polynomial first appeared in 1974 in the work of Matiyasevich \cite{matiyasevich2007criterion}.
Its use was popularized by the paper of Alon and Tarsi \cite{alon1992colorings} (who discovered it independently).

For a vertex coloring $c$ of graph $G$, 
we can see that $f_G(c(v_1),c(v_2),\dots, c(v_n)) \neq 0$ if and only if the coloring  is proper (where the colors are interpreted as elements of some ring, say $\mathbb{Z}$).
Numerous applications of the Alon and Tarsi method inspired Jensen and Toft \cite{jensen1995graph} to define the Alon-Tarsi number of a graph. 
\begin{defn}[\cite{jensen1995graph}]
\emph{The Alon-Tarsi number of a graph $G$}, denoted by $\AT(G)$, is defined as
\[
    \AT(G) = \alpha(f_G) + 1.\
\]
\end{defn}

The authors of both seminal papers \cite{alon1992colorings,matiyasevich2007criterion} observed that  for any graph $G$
\[
    \ch(G) \leq \AT(G).
\]
Alon and Tarsi proved the above proposition by applying Combinatorial Nullstelensatz to the graph polynomial $f_G$.
Schauz \cite{schauz2010paintability} observed that this line of argument can be extended to the case of on-line choosability obtaining
\[
    \ch_{OL}(G) \leq \AT(G).
\]

The following simple observation allows us to use triangular graphs when proving upper bounds for the Alon-Tarsi number for planar graphs.
\begin{fact}
\label{fact:ATmon}
If $G$ is a subgraph of $H$, then 
$$AT(G) \leq AT(H).$$
\end{fact}
\begin{proof}
Graph polynomial $f_G$ divides graph polynomial $f_H$. 
Therefore  $\alpha(f_G) \leq \alpha(f_H),$ and $AT(G) \leq AT(H).$
\end{proof}

The next lemma is also from the work of Alon and Tarsi \cite{alon1992colorings}.
A graph is called \emph{even (odd)}, depending on the parity of the number of its edges.
Recall that a directed graph $F$ is \emph{an Eulerian subgraph of an orientation $O$ of graph $G$} if $F$ is a subgraph of $O$ and for every vertex in $F$, the in-degree in $F$ is equal to the out-degree in $F$.
\begin{lem}[\cite{alon1992colorings}] \label{lem:org_even}
Consider a graph $G$ and its orientation $O$. 
Let $k$ be the maximum in-degree in this orientation. 
If the number of even Eulerian subgraphs of $O$ is different from the number of odd Eulerian subgraphs of $O$ then
\[
    \AT(G)-1 \leq k.
\]
\end{lem}
A generalized version of the above lemma is proved in the next section.

The coloring number of a graph $G$, denoted by $\col(G)$ is the minimum number $k$ such that there exists an acyclic orientation of the edges of $G$ for which in-degree of every vertex is at most $k-1$.
Clearly, by Lemma \ref{lem:org_even} we have
\[
    \AT(G) \leq \col(G).
\]

\subsection {Extension to augmented graphs} 
We define an extended version of the graph polynomial for graphs with augmented edges. 

\begin{defn} 
An \emph{Augmented orientation} is a tuple $(G, D, O)$ where $G$ is a graph, 
$O$ is an orientation of $G$, and $D:E\to \mathbb{N}$ assigns positive strengths to the edges of $G$.
An edge with strength $k$ will be called a \emph{$k$-edge}. 
An edge with strength 2 will be called a \emph{double edge}. 
\emph{Augmented in-degree} (resp. \emph{augmented out-degree}) is defined as the sum of strengths of the ingoing (resp. outgoing) edges.
\end{defn}

Let $(G,D,O)$ be an augmented orientation.
We define a \emph{graph polynomial} for an augmented graph as follows.
\[
    W_{G,D}(x_1,\dots,x_n) = \prod_{i<j \wedge e=\{v_i, v_j\} \in E}(x_i^{D(e)} - x_j^{D(e)}).
\]
\begin{rem}
Observe that $W_{G,D}$ is different from the polynomial $f_{G'}$ of a multigraph $G'$, constructed by replacing every $k$-edge of $(G,D,O)$ with $k$ parallel edges.
Note also that $W_{G,D}= f_G$ if every edge has strength $1$.
\end{rem}
\smallskip
\noindent
Using identity
\[
    (x_a^{D(e)} - x_b^{D(e)}) 
    = (x_a - x_b)\left(\sum_{i=0}^{D(e)-1} x_a^{D(e)-1 - i}x_b^i \right),
\]
polynomial $W_{G,D}$ can be rewritten into
\begin{align*}
    W_{G,D}(x_1,\dots,x_n) & = \left(\prod_{i<j \wedge e=\{v_i, v_j\} \in E}(x_i - x_j) \right) \cdot P(x_1,\dots,x_n)
    \\ & = f_G(x_1,\dots,x_n) \cdot P(x_1,\dots,x_n),
\end{align*}
where $P$ is a nonzero polynomial. 
From this we conclude the following fact
\[
    \AT(G)-1 = \alpha(f_G) \leq \alpha(W_{G,D}).
\]

\begin{defn}[Eulerian structure] 
A graph $F$ without isolated vertices is \emph{an Eulerian structure in augmented orientation $(G,D,O)$} if $F$ is a subgraph of $G$ and for every vertex $v$ 
its augmented in-degree is equal to its augmented out-degree in augmented orientation $(F,D_F,O_F)$, where $O_F$ is $O$ restricted to $F$ and $D_F$ is $D$ restricted to $F$.
\end{defn}
The proof of the next lemma is a straightforward extension of the proof of Alon and Tarsi \cite{alon1992colorings} to the case of graphs with augmented edges.
\begin{lem} \label{wgd_atg_relation}
Consider augmented orientation $(G,D,O)$. 
Let $k$ be the maximum augmented in-degree in this augmented orientation. 
If the number of even Eulerian structures in $(G,D,O)$ is different than the number of odd Eulerian structures in $(G,D,O)$ then
\[
    \AT(G)-1 = \alpha(f_G) \leq \alpha(W_{G,D}) \leq k.
\]
\end{lem}
\begin{proof}
We define function $S$ from the set of orientations of $G$ to the set of monomials that can potentially occur in $W_{G,D}$. 
Given an orientation, from every term $(x_i^D(e)-x_j^D(e))$ of the product defining $W_{G,D}$, 
 we choose $x_i^D(e)$ if the orientation directs edge $e = v_iv_j$ to $v_i$; otherwise, we choose $-x_j^D(e)$.
Then, the value of $S$ on the orientation is the product of the chosen terms. 

We can see that
\[
    W_{G,D} = \sum_{R \text{ - orientation of $G$}} S(R).
\]
The degree of $x_i$ in $S(R)$ is exactly the augmented in-degree of $v_i$ in (augmented) orientation $(G,D,R)$.
Consider two orientations $R_1, R_2$ such that $|S(R_1)| = |S(R_2)|$
(i.e. the monomials produced from $R_1$ and $R_2$ differ at most by sign). 
Let $A$ be the set of edges that are oriented differently in $R_1$ than in $R_2$. 
We can notice that $A$ induces Eulerian structure in both $(G,D,R_1)$, $(G,D,R_2)$ (because augmented in-degree of every vertex is the same in $R_1$ and $R_2$). 
Moreover, the sign of $S(R_1)$ is equal to the sign of $S(R_2)$ if and only if the number of edges in $A$ is even.
Therefore, for every Eulerian structure $B$ in $(G,D,R_1)$, when we change the orientation of the edges belonging to $B$, we get an orientation $R_3$ for which $|S(R_1)| = |S(R_3)|$. 
Let $CO(R), CE(R)$ be respectively, the number of Eulerian structures in $(G,D,R)$ with odd and even number of edges. 

The above discussion allows us to conclude that, whenever $CO(O) \not = CE(O)$, we have
\[
    \sum_{R : |S(R)| = |S(O)|}S(R) \not = 0.
\]
Since the maximum augmented in-degree of $(G,D,O)$ is at most $k$, we obtain that $\alpha(S(R)) \leq k$. 
That implies
\[
    \AT(G)-1 = \alpha(f_G) \leq \alpha(W_{G,D}) \leq k.
\]
\end{proof}

\subsection {Schnyder woods}
In this section, we recall the definitions and basic properties of Schnyder labellings.
We are going to use them in the main proof.
Considered theorems are trivial for graphs with fewer than 3 vertices. In this section, we assume that the planar graph have at least 3 vertices. 

\begin{defn}
 A \emph{triangular graph} is a plane graph whose faces are triangles.
\end{defn}
Clearly, every planar graph is a subgraph of some triangular graph on the same vertex set.
By monotonicity of the Alon-Tarsi number (Fact \ref{fact:ATmon}) it is enough to prove the upper bound for triangular graphs.
Observe also that triangular graphs are 2-connected.

\subsubsection{3-orientations and realizers}
The orientations studied within the framework of Schnyder labellings need to be reversed in order to be useful for the Alon-Tarsi method.
For the sake of consistency, we reverse the orientations when quoting results from that field.
For example, in the original version of the following definition from \cite{schem}, every vertex has \emph{outdegree} one in each of the sets and the counterclockwise order of the edges
incident to $v$ is: outgoing edge of $T_r$, incoming edges of $T_b$,
outgoing edge of $T_g$, incoming edges of $T_r$, outgoing edge of $T_b$, incoming edges of $T_g$.
\begin{defn}{\cite{schem}}
A \emph{realizer} of a triangular graph
$G$ is a tuple of three sets of oriented edges $(T_r,T_g,T_b)$ such that,
after ignoring edge orientations, these three sets form a partition of the interior edges of $G$ and 
such that for each interior vertex $v$ of $G$ it holds:
\begin{enumerate}
    \item $v$ has in-degree one in each of $T_r, T_g, T_b$.
    \item The counterclockwise order of the edges
incident to $v$ is the following: incoming edge of $T_r$, outgoing edges of $T_b$,
incoming edge of $T_g$, outgoing edges of $T_r$, incoming edge of $T_b$, outgoing edges of $T_g$.
\end{enumerate}
\end{defn}
Note that in the above definition, for any interior vertex $v$ there might be no outgoing edges in any of the sets.

Schnyder proved in \cite{schem} that every triangular graph has a realizer.
In addition to partitioning the interior edges into three sets, the realizer also orients the interior edges of $G$.
In this orientation, every interior vertex has in-degree exactly $3$.
Such orientations of the interior edges of a triangular graph are called \emph{internal $3$-orientations}.
It is easy to check that, by the Euler formula, in every such orientation, no interior edge is directed toward an exterior vertex.
De Fraysseix and de Mendez observed in \cite{FraPOM01} that an internal $3$-orientation uniquely determines a realizer.
Therefore, there is a bijection between the set of realizers and the set of internal $3$-orientations of a graph $G$. 
Both of these perspectives are going to be useful for our needs.
(Note that we are ignoring edges of the outer triangle.)
We use the following observation from \cite{schem}.
\begin{prop}[Theorem 4.5 in \cite{schem}]\label{properties_of_labelling}
Let $G$ be a triangular graph with realizer $(T_r,T_g,T_b)$. 
Then $T_r, T_g, T_b$ are trees and each of $T_r, T_g, T_b$ spans all interior vertices of $G$.
\end{prop}
Every large enough triangular graph admits a number of different realizers.
Natural operation of inverting the edges of a directed triangle in an internal 3-orientation allows to transform one internal 3-orientation into another. 
Starting from this notion, Brehm \cite{brehm_orie} studied the graph of orientations and discovered that it is naturally organized into a structure of a distributive lattice.
The top and bottom elements of that lattice are the unique orientation without respectively clockwise and counter-clockwise cycles.
\begin{prop}[Theorem 1.3.3 and Lemma 1.7.7 in \cite{brehm_orie}]
\label{clockwise}
For every triangular graph, there exists exactly one internal 3-orientation in which all directed cycles are oriented counterclockwise.
\end{prop}
The counterclockwise orientation from the previous proposition plays a central role in our proof.

\subsubsection{Warm-up example}
We illustrate how the basic results on the realizers of planar graphs can be used to derive an alternative proof of the following strengthening of one of the results from \cite{kim2019alon}. 
(The actual bound from \cite{kim2019alon} was for the Alon-Tarsi number.)
\begin{thm}[\cite{kim2019alon}]
\label{thm:G-F}
For every planar graph $G$ there exists a forest $F$ such that 
\[
    \col(G-F) \leq 3.
\]
\end{thm}
We are going to use the following version of Theorem 4.6 from \cite{schem}.
\begin{prop}\label{prop:no_cycle_in_2}
Consider a triangular graph $G = (V,E)$ with realizer $(T_r,T_g,T_b)$. 
Then, the oriented graph $(V, T_r \cup T_b)$ is acyclic.
\end{prop}
\begin{proof}[Proof of Theorem \ref{thm:G-F}]
By the monotonicity of the coloring number, it is enough to prove the theorem for triangulations.
Let $G=(V,E)$ be a triangulation and let $(T_r, T_g, T_b)$ be a realizer of $G$ (note that realizers exist by the results of Schnyder \cite{schem}).
By Proposition \ref{prop:no_cycle_in_2}, oriented graph $H=(V, T_r \cup T_g)$ is acyclic.
Moreover, by the definition of a realizer, every interior vertex has in-degree 2 in $H$ and every exterior vertex has in-degree 0.
Therefore, we can add to $H$ the edges of the outercycle of $G$ and orient them in such a way that the resulting graph $H'$ is still acyclic.
Then, the in-degrees of all vertices $H'$ are still at most 2.
This implies that $\col(H')$ is at most 3.
Note that after dropping the orientation, the edges of $H'$ are precisely the edges of $G-T_b$.
Let $F$ be the set of edges $T_b$ with dropped orientations.
The orientation of $H'$ proves that 
\[
    \col(G-F) \leq 3.
\]
\end{proof}

\section{Counterclockwise orientations and Alon-Tarsi numbers of planar graphs}
\label{sec:proof}

We recall a few more definitions and a few results from \cite{schem}.
For a realizer $(T_r,T_g,T_b)$ of a triangular graph $G$, a \emph{colored path} is a path from an interior vertex of $G$ to an exterior vertex of $G$ that has only edges from one of the sets $T_r, T_g, T_b$.
Every tree induced by sets $T_r, T_g, T_b$ contains only one exterior vertex of $G$ (consequence of Theorem 4.5 in \cite{schem}).
We call these vertices \emph{roots} and denote them by $v_r, v_g, v_b$.
For every vertex, there are 3 unique colored paths that start at that vertex \cite{schem}.

The colored paths in $T_r, T_g, T_b$ are correspondingly called the \emph{red path, green path}, and \emph{blue path}.
Colored paths that contain an interior vertex $v$ of $G$ as one of the ends are correspondingly denoted by $P_r(v), P_g(v), P_b(v)$. 
The path $P_i(v)$ ends in $v_i$ for $i \in \{r,g,b\}$.
Under the orientation given by the realizer, path $P_r(v)$ is a directed path from $v_r$ to $v$.
For an interior vertex $v$ of $G$ we see that $P_r(v), v_rv_b, P_b(v)$ forms a simple (i.e., not directed) cycle $C$. 
The subgraph of $G$ bounded by $C$ is called \emph{the green region of $v$} and is denoted by $R_g(v)$. . 
Regions of other colors are defined in an analogous way.
For interior vertices $u,v$ of $G$, Lemma 5.2 in \cite{schem} implies that
\begin{equation}
\label{eq:reg_sub}
    u \in R_g(v) \implies R_g(u) \text{ is a subgraph of } R_g(v).
\end{equation}
In other words, green regions are partially ordered by the relation of being a subgraph.

We are ready to prove the main ingredient of our argument.
\begin{prop} \label{prop:mainw}
For every triangular graph $G$ there exists an augmented orientation $(G,D,O)$
without nonempty Eulerian structures and with the maximum augmented in-degree at most 4.
\end{prop}
\begin{proof}
From \ref{clockwise} we know that there exists an internal 3-orientation $L_G$ of $G$ in which all cycles are oriented counterclockwise.
Such an orientation 
can be extended to external edges in such a way that all directed cycles are oriented counterclockwise and the outer face is not a cycle.
Let $O$ be such an extension.
We have that the orientation $L_G$ corresponds to some realizer $(T_r$, $T_b$, $T_g)$. 
Let $D$ be an assignment of strengths to the edges of $G$ such that the edges of $T_r$ have strength $2$ and all the other edges have strength $1$. 
Then $(G,D,O)$ is an augmented orientation.

Suppose for a contradiction that there exists a nonempty Eulerian structure $H$ in $(G,D,O)$.
Consider a vertex $v$ of $H$ with a minimal green region $R$. 
We know that $H$ does not contain any other vertex or edge of $R$.

First, we consider the case where $v$ is an interior vertex of $G$.
There are three edges directed to $v$ colored red, green, and blue, respectively (colors are given by the Schynder labelling). 
We know that blue and red edges oriented towards $v$ cannot be in $H$, because $v$ is a vertex with a minimal green region (see property (\ref{eq:reg_sub})).
Then, only the green edge can enter $v$ in $H$. 
$H$ is nonempty, therefore, that edge belongs to $H$.
Then, the augmented in-degree of $v$ in $H$ is equal to 1, 
    so the augmented out-degree of $v$ in $H$ is also equal to 1.
Red edges contribute $2$ to augmented out-degrees so there cannot be any outgoing red edges from $v$ in $H$.
In $H$, there cannot be any green edges outgoing from $v$, due to the minimality of the green region of $v$.
Then the only edge of $H$ that is oriented away from $v$ must have a blue color. 
Since $H$ is a Eulerian structure, it contains a simple directed cycle $C$ that contains $v$.
We see that cycle $C$ must contain one green edge directed to $v$ in $O$ and one blue edge directed from $v$ in $O$.
Furthermore, $C$ does not contain vertices from the green region of $v$. 
Therefore, $C$ must be oriented clockwise.
But in orientation $O$ all directed cycles are oriented counterclockwise. 
This is a contradiction.

The other case, where $v$ is an outer-vertex, is simpler. 
The vertices on the outer face do not have ingoing edges from the interior vertices, so the only cycle containing them could be the outer face, but the outer face is not a cycle.
Therefore, $v$ cannot belong to a directed cycle, which contradicts the fact that every Eulerian structure contains a directed cycle.
\end{proof}

Proposition \ref{prop:mainw} enables an alternative way of proving the main result of Zhu from \cite{zhuat}.
\begin{thm}
\label{thm:planarAT}
    The Alon-Tarsi number of a planar graph is at most 5.
\end{thm}
\begin{proof}
Let $G$ be a triangulation of a plane graph $F$. 
From Lemma \ref{clockwise} there exists an internal 3-orientation $O$ in which all directed cycles are oriented counterclockwise.
Then, by Proposition \ref{prop:mainw}, there exists an augmented orientation $(G,D,O)$ without nonempty Eulerian structures. 
The maximal in-degree of a vertex in $(G,D,O)$ is at most 4. 
From Lemma \ref{wgd_atg_relation} we obtain
\[
    4 \geq \alpha(W_{G,D}) \geq \alpha (f_G) = \AT(G) - 1.
\]
Therefore, $AT(F)$ is at most 5 as well.
\end{proof}

\section{Counterclockwise orientation in Thomassen's proof}
\label{sec:Thom}

The original proof of Theorem \ref{thm:planarAT} from \cite{zhuat} followed the ideas of the famous proof of Thomassen of the fact that the choice number of a planar graph is at most 5 \cite{thom5c}.
It is interesting that the list coloring algorithm that is implicitly given in Thomassen's work can be easily modified to construct a useful augmented orientation.
We describe below the modified procedure.

\subsection{Algorithm description}

The input of the procedure is a 2-connected near triangulation with two distinguished vertices $v_1,v_2$ that are clockwise consecutive on the outer cycle.
The edge between the distinguished vertices is already oriented; the other edges are not.
The procedure constructs an orientation of all the edges of the given graph.
Its behaviour depends on whether the outercycle has a chord.
The cases are described in the following paragraphs.
\bigskip

\textbf{The outer cycle has a chord -- recursive step}
    
    Let $v_av_b$ be a chord of the outercycle.
    It divides the outer cycle into two cycles $C_1$ and $C_2$,
        where $C_1$ contains both $v_1, v_2$.
    Cycles $C_1$ and $C_2$ together with their interiors, determine two subgraphs of $H$ denoted by $H_1, H_2$.
    The only common vertices of $H_1$ and $H_2$ are $v_a$ and $v_b$.
    The procedure is run recursively first in subgraph $H_1$ with vertices $v_1, v_2$, 
        and then in subgraph $H_2$ with distinguished vertices $v_b, v_a$ (note that edge $v_av_b$ has been oriented in the run of the procedure on $H_1$).
\bigskip

\textbf{There is no chord on the outer cycle -- orienting step}
    If there is no chord on the outer cycle, we focus on the vertex that immediately follows $v_2$ in clockwise order of the outercycle.
    We denote that vertex by $v_3$ and call it \emph{the central vertex} of this step.
    In this step, the procedure orients and assigns strengths to the edges of the current graph adjacent to $v_3$ according to the following rules: 
    \begin{itemize}
        \item edges of the outer cycle are oriented towards $v_3$ and are given strength 1,
        \item all the other edges are oriented away from $v_3$ and are given strength 2.
    \end{itemize}
    Finally, vertex $v_3$ is removed from the graph.
    If there are still some vertices beside $v_1,v_2$ left, the procedure is recursively called on the remaining graph with the same distinguished vertices $v_1,v_2$.

\subsection{Constructed orientation}
    The procedure being recursive has to be defined on near triangulations.
    We are interesting however in running it on triangular graphs. 
    To start the procedure, we choose two vertices of the outer triangle as distinguished vertices.
    Then, the first step of the procedure is the orienting step in which the central vertex is the third vertex of the outer triangle.
    Then, in the run of the procedure, the following invariants are kept.
    Note that in some sense they mimic Thomassen's conditions on the lengths of lists.
    \begin{enumerate}
        \item[(I1)] At the beginning of each step, no inner edge of the current near triangulation is oriented.
        \item[(I2)] At the beginning of each but the first step, every not-distinguished vertex of the current outer cycle has exactly one ingoing edge. 
                    Moreover, that edge does not lie on the outer cycle and is doubled.
       
        \item[(I3)] For every vertex $v$, after the orienting step in this vertex is removed, all the edges adjacent to $v$ will be oriented away from $v$.
    \end{enumerate}
    
    We make a few observations about the orientation constructed by the procedure.
    \begin{prop}
        The procedure run on a triangular graph constructs an internal 3-orientation.
    \end{prop}
    \begin{proof}
    Consider an internal vertex $v$.
    Invariant (I2) guarantees that at the beginning of the (orienting) step, when  vertex $v$ is removed from the graph, it has exactly one incoming double edge.
    During that step, exactly two edges are oriented towards $v$.
    Altogether, its in-degree becomes 3.
    By (I3), if an edge adjacent to $v$ is oriented in one of the later steps, it is always oriented away from $v$.
    Therefore, the in-degree of $v$ stays 3 until the end of the procedure.
    \end{proof}

    \begin{prop}
    \label{prop:no-cwo-triangles}
        The orientation constructed by the procedure does not contain a clockwise oriented triangle.
    \end{prop}
    \begin{proof}
    Suppose for a contradiction, that a clockwise oriented triangle has been constructed.
    
    Consider the orienting step of the procedure in which the first vertex $v$ of the triangle has been removed.
    In the final orientation, there are precisely 3 edges oriented towards $v$.
    Moreover, at the step in which $v$ was removed, exactly one of these edges was incoming from a vertex that was removed earlier.
    By the choice of $v$, this edge cannot belong to the triangle.
    
    Note also that beside the first two distinguished vertices, any vertex that becomes distinguished in the second call of a recursive step must first be removed in the first call of this step.
    This implies that no edges incoming to $v$ from a distinguished vertices of the current step can belong to the triangle.
    
    Only one edge incoming to $v$ is still a candidate for an edge of the triangle.
    It is the edge from the (clockwise) next vertex of the outer cycle $v'$.
    Moreover the current outer cycle cannot be a triangle since then $v'$ would be distinguished and hence removed before $v$.
    Let $v''$ be the third vertex of the triangle.
    Vertex $v''$ can not be an interior vertex of the current near triangulation since then the triangle would not be clockwise oriented.
    It cannot also be the unique vertex with the edge currently directed towards $v''$ since that vertex has been removed before $v$.
    Therefore, $v''$ must belong to a second recursive call of an earlier recursive step of the procedure.
    However, in such a case, vertex $v'$ is going to be removed before $v''$ which (by invariant (I3)) implies that the edge between $v'$ and $v''$ is going to be directed towards $v''$ which contradicts the fact that the triangle is oriented.
    \end{proof}

    \begin{prop}
        Edges doubled by the procedure run on a triangular graph form one of the trees of the realizer corresponding to the constructed internal 3-orientation.
    \end{prop}
    \begin{proof}
        Discussed invariants imply that double edges always form a forest that spans all internal vertices.
        Moreover, when the procedure is run on a triangular graph, that forest is, in fact, a tree that is rooted in the unique vertex of the outer triangle that has not been distinguished in the first call.
        It remains to verify that all the doubled edges indeed belong to one of the trees of the realizer.
        It is easily verified that in the orienting step, the doubled edges are exactly all the outgoing edges between two incoming edges that are not doubled.
        That implies that they all get the same color (i.e. are going to belong to the same tree of the realizer).
        Moreover, the color is the same as the color of the unique ingoing doubled edge.
    \end{proof}

    It has been observed in \cite{brehm_orie} (Corollary 1.5.2) that whenever there is a clockwise oriented cycle in an orientation of a triangular graph, this orientation also contains a clockwise oriented triangle.
    Therefore, Proposition \ref{prop:no-cwo-triangles} implies that the orientation constructed by the procedure is in fact the counterclockwise internal 3-orientation in which the edges of one of the trees of corresponding realizers are doubled.
    As we already explained (Proposition \ref{prop:mainw} and the proof of Theorem \ref{thm:planarAT}), that orientation certifies that $\AT(G)\leq 5$.
    Interestingly, the above procedure can be also viewed as a specific realisation of the algorithm of Brehm from \cite{brehm_orie} designed for generating realizers of planar graphs.

\section{Final thoughts}
We showed that Zhu's strengthening of the result of Thomassen can be derived from the independent developments on the realizers of planar graphs started by Schnyder.
Although stemming from different lines of research, these two proofs are not really different.
We showed that a natural modification of Thomassen's procedure can be used to construct an orientation that certifies that the Alon-Tarsi number of the graph is at most 5.

A slight technical generalization of the method of Alon and Tarsi was used in our proof.
It has a property that, while it allows proving $\AT(G)\leq 5$, it does not help to find a monomial of $f_G$ that certifies that property.
At the same time, after careful strengthening of some edges, the problem of counting Eulerian structures becomes much easier.
Indeed, in some sense, all nontrivial Eulerian subgraphs of the orientations corresponding to monomials of $f_G$ certifying $\AT(G)\leq 5$ cancel themselves.
We expect that augmented orientations may also be useful in analysing graph polynomials  for other graph classes.
Let us also note that our modification of $f_G$ into $W_{G,D}$ by strengthening edges is just one of the possible options.
Polynomial $W_{G,D}$ used in our proof can be viewed as $f_G$ multiplied by another  polynomial defined by a specific spanning tree of the underlying graph.
It is tempting to look for other useful examples of such multipliers that depend on other substructures of the studied graphs.

The bound from Theorem \ref{thm:planarAT1} has also recently been given a short proof by Gu and Zhu \cite{gu2022alon}.
They managed to greatly simplify the original argument with a technique that is somewhat similar to the extension of Alon-Tarsi method used in the current paper.
As these two ideas seem to be related, it is rational to take both perspectives into account when considering potential generalizations of graph polynomials described above.
That paper also contains a simplified proof of a recent result of Grytczuk and Zhu \cite{grytczuk2020alon} that every planar graph $G$ contains a matching $M$ such that $\AT(G-M)\leq 4$.
It would also be interesting to describe this matching in terms of Schnyder realizers.

\bibliographystyle{plain}
\bibliography{ref}

\end{document}